\documentclass[a4paper]{amsart}

\usepackage[utf8]{inputenc}
\usepackage[toc]{appendix}

\usepackage[english]{babel}

\usepackage{graphicx}
\usepackage{subfigure}
\usepackage[justification=centering]{caption}

\usepackage{amsmath,a4wide}
\usepackage{amsfonts}
\usepackage{amssymb}
\usepackage{amsthm}
\usepackage{mathtools}
\usepackage{amsaddr}

\usepackage{enumerate}

\usepackage{hyperref}
\usepackage{url}

\numberwithin{equation}{section}
\mathtoolsset{showonlyrefs}
\providecommand{\norm}[1]{\big\lVert#1\big\rVert}
\newcommand{\R}{\mathbb{R}}

\newcommand{\Rd}[1][d]{\mathbb{R}^{#1}}
\newcommand{\Z}{\mathbb{Z}}
\newcommand{\N}{\mathbb{N}_+}
\newcommand{\Lt}[1][d]{L^2(\R^{#1})}

\newcommand{\G}{\mathcal{G}}
\newcommand{\F}{\mathcal{F}}

\newcommand{\V}{\mathcal{V}}

\newcommand{\Lp}[1]{L^#1}

\renewcommand{\l}{\lambda}
\renewcommand{\L}{\Lambda}

\newcommand{\Jhat}{\widehat{J}}
\newcommand{\Mhat}{\widehat{M}}
\newcommand{\Vhat}{\widehat{V}}

\DeclareMathOperator*{\essinf}{ess\,inf}
\DeclareMathOperator*{\esssup}{ess\,sup}

\newcommand{\J}{
	\left(
		\begin{array}{rc}
		0 & I\\
		-I & 0
		\end{array}
	\right)
}
	
\newcommand{\VP}{
	\left(
		\begin{array}{rc}
		I & 0\\
		P & I
		\end{array}
	\right)
}

\newcommand{\ML}{
	\left(
		\begin{array}{ll}
		L^{-1} & 0\\
		0 & L^T
		\end{array}
	\right)
}
	
\newcommand{\Sblock}{
	\left(
		\begin{array}{cc}
		A & B\\
		C & D
		\end{array}
	\right)
}

\newcommand{\SW}{
	\left(
		\begin{array}{cc}
		L^{-1}Q & L^{-1}\\
		PL^{-1}Q-L^T & PL^{-1}
		\end{array}
	\right)
}

\theoremstyle{plain}
\newtheorem{theorem}{Theorem}[section]

\theoremstyle{plain}
\newtheorem{corollary}[theorem]{Corollary}

\theoremstyle{plain}

\theoremstyle{plain}
\newtheorem{proposition}[theorem]{Proposition}

\theoremstyle{definition}
\newtheorem{definition}[theorem]{Definition}

\theoremstyle{remark}

\theoremstyle{remark}

\theoremstyle{definition}

\begin{document}
\title[Metaplectic Operators and an Extension of a Result of Lyubarskii and Nes]{On the Parity under Metapletic Operators and an Extension of a Result of Lyubarskii and Nes}
\author[Markus Faulhuber]{Markus Faulhuber}
\address{NuHAG, Faculty of Mathematics, University of Vienna, Oskar-Morgenstern-Platz 1, 1090 Vienna, Austria}
\email{markus.faulhuber@univie.ac.at}
\thanks{
	The author was supported by the Erwin--Schrödinger Program of the Austrian Science Fund (FWF): J4100-N32. Parts of this work have already been established during the first stage of the program, which the author spent with the Analysis Group at the Department of Mathematical Sciences, NTNU Trondheim, Norway
	%The computational results presented have been achieved (in part) using the Vienna Scientific Cluster (VSC)
}

\begin{abstract}
	In this work we show that if the frame property of a Gabor frame with window in Feichtinger's algebra and a fixed lattice only depends on the parity of the window, then the lattice can be replaced by any other lattice of the same density without losing the frame property. As a byproduct we derive a generalization of a result of Lyubarskii and Nes, who could show that any Gabor system consisting of an odd window function from Feichtinger's algebra and any separable lattice of density $\tfrac{n+1}{n}$, $n \in \N$, cannot be a Gabor frame for the Hilbert space of square-integrable functions on the real line. We extend this result by removing the assumption that the lattice has to be separable. This is achieved by exploiting the interplay between the symplectic and the metaplectic group.
\end{abstract}


\maketitle

\section{Introduction}\label{sec_intro}

In this article we extend results derived by Lyubarskii and Nes, who proved that for odd functions in Feichtinger's algebra $S_0(\R)$ and separable lattices of rational density $ \tfrac{n+1}{n}$, $n \in \N$, the corresponding Gabor system is not a frame for $\Lt[]$ \cite{LyubarskiiNes_Rational_2013}. We show that the assumption of the separability of the lattice is not necessary and that their result holds for arbitrary lattices of density $\tfrac{n+1}{n}$.

We will formulate our results as well as most of the theory for $\Lt$. However, for the case $d > 1$ the lack of knowledge about frame sets and obstructions of Gabor systems is too big as to formulate a general result like Corollary \ref{cor_LyuNes} below. Of course, it is possible to produce results for $d > 1$ by using tensor products of Gabor systems and the results of Bourouihiya \cite{Bou08}, but this still leaves quite a big gap to the generic case.

The main tools in this work are metaplectic operators, which are a certain class of unitary operators acting on $\Lt$, and their interplay with the symplectic group $Sp(d)$, a subgroup of $SL(\R,2d)$. In order to generalize the result in \cite{LyubarskiiNes_Rational_2013}, we will prove and use the following result.

\begin{theorem}\label{thm_main}
	Let $g \in S_0(\Rd)$ be either an even or an odd function and consider the (hyper-cubic) lattice $\delta^{-1/2d} \, \Z^{2d}$ of density $\delta > 1$. If for all either even or odd $g \in S_0(\Rd)$ the Gabor system $\G \left(g, \delta^{-1/2d} \, \Z^{2d} \right)$ is (not) a Gabor frame for $\Lt$, then the Gabor system $\G \left(g, \delta^{-1/2d} \, S \Z^{2d} \right)$ is also (not) a Gabor frame for $\Lt$ for any symplectic matrix $S \in Sp(d)$.
\end{theorem}

The above result says the following; if we can show that for a given lattice density the frame property of a Gabor system only depends on the parity of the window, then the frame property only depends on the lattice density and not on the geometry of the (symplectic) lattice.

In particular, for $d = 1$ we will use the above theorem to exclude all lattices of density $\tfrac{n+1}{n}$, $n \in \N$, from the frame set of any odd function in $S_0(\R)$. If we want to exclude all lattices of a given density from the frame set of any even function in $S_0(\R)$, we are left with the case of critical density. The reason why we cannot exclude other densities relies on the fact that any redundant Gaussian Gabor system is a frame for $\Lt[]$ as was shown by Lyubarskii \cite{Lyu92}, Seip \cite{Sei92} and Seip and Wallsten \cite{SeiWal92}.

Assuming Theorem \ref{thm_main}, we can easily extend the main result of Lyubarskii and Nes in \cite{LyubarskiiNes_Rational_2013}.
\begin{corollary}\label{cor_LyuNes}
	Let $g \in S_0(\R)$ be an odd function and consider a lattice
	\begin{equation}
		\L = \delta^{-1/2} S \Z^2, \quad S \in Sp(1) = SL(\R,2).
	\end{equation}
	Then the Gabor system $\G(g,\L)$ is not a frame whenever $\delta = \tfrac{n+1}{n}$, $n \in \N$.
\end{corollary}

This work is structured as follows:
\begin{itemize}
	\item[-] In Section \ref{sec_Gabor} we recall the basic facts about time-frequency analysis and Gabor systems and define the function space $S_0(\Rd)$.

	\medskip

	\item[-] In Section \ref{sec_symp_meta} we introduce the symplectic and the metaplectic group and state how they can be used to deform Gabor systems without changing their frame property.

	\medskip

	\item[-] In Section \ref{sec_density} we state the result of Lyubarskii and Nes and show that it does not only hold for separable lattices, but arbitrary 2-dimensional lattices.

	\medskip
	
	\item[-] In Section \ref{sec_critical} we study Gabor systems of even and odd functions from $S_0(\R)$ at critical density and derive a short proof that these systems cannot be frames.
	
	\medskip	

	\item[-] We close this work with questions which we consider to be interesting open problems.
\end{itemize}

\section{Gabor Systems and Frames}\label{sec_Gabor}
In this section we introduce the basics of time--frequency analysis. The notation is mainly according to the textbook of Gröchenig \cite{Gro01}.

A Gabor system consists of time-frequency shifted copies of a window function $g$ which is usually chosen from a subspace of $\Lt$. In this work, we consider windows from Feichtinger's algebra $S_0(\Rd)$, which we define later in this section. We denote the time-frequency shift operator by $\pi(\l)$, where $\l = (x, \omega)$ is a point in the time-frequency plane $\Rd \times \widehat{\Rd}$. Here, $\widehat{\Rd}$ is the dual group of $\Rd$, or group of characters, which is isomorphic to $\Rd$. Hence, we identify $\Rd \times \widehat{\Rd}$ with $\R^{2d}$. The explicit action of a time-frequency shift on a window $g$ is given by
\begin{equation}
	\pi(\l) g(t) = M_\omega T_x g(t) = e^{2 \pi i \omega \cdot t} g(t-x), \qquad t,x,\omega \in \Rd,
\end{equation}
where $\omega \cdot t$ denotes the Euclidean inner product on $\Rd$. The operators $T_x$ and $M_\omega$ are called time-shift (or translation) operator and frequency-shift (or modulation) operator, respectively. A tool taking a central role in the theory of Gabor analysis is the short-time Fourier transform (STFT) which we denote in the following way;
\begin{equation}
	\V_g f(\l) = \int_{\Rd} f(t) \overline{g(t-x)} e^{-2 \pi i \omega \cdot t} \, dt = \langle f , \pi(\l) g \rangle, \qquad \l = (x,\omega) \in \R^{2d}.
\end{equation}
Here, $\overline{g}$ is the complex conjugate to $g$ and $\langle f,g\rangle$ is the $\Lt$ inner product of $f,g \in \Lt$. A Gabor system is now defined as the following set of functions;
\begin{equation}
	\G(g, \L) = \{ \pi(\l) g \mid \l \in \L \subset \R^{2d} \}.
\end{equation}
Throughout this work, we assume that $\L$ is a lattice, i.e., a discrete, co-compact subgroup of $\R^{2d}$. Any lattice can be written as
\begin{equation}
	\L = \delta^{-1/2d} M \Z^{2d}, \qquad M \in SL(\R, 2d), \; \delta > 0.
\end{equation}
The parameter $\delta$ is called the density of the lattice, which gives the average number of lattice points per unit volume.

A Gabor system is a frame for $\Lt$ if and only if the frame inequality is fulfilled. This means, there exist constants $0 < A \leq B < \infty$ such that
\begin{equation}\label{eq_frame}
	A \norm{f}^2 \leq \sum_{\l \in \L} \left| \langle f , \pi(\l) g \rangle \right|^2 \leq B \norm{f}^2, \qquad \forall f \in \Lt.
\end{equation}
The frame operator associated to the Gabor system $\G(g,\L)$ is given by
\begin{equation}
	S_{g,\L} f = \sum_{\l \in \L} \langle f , \pi(\l) g \rangle \, \pi(\l) g.
\end{equation}
If $\G(g,\L)$ is a frame, i.e., the frame inequality is fulfilled for positive constants $A$ and $B$, then the frame operator is bounded and has a bounded inverse. Hence, any $f \in \Lt$ can be stably reconstructed from its measurements with respect to the Gabor system;
\begin{equation}
	f = \sum_{\l \in \L} \langle f , \pi(\l) g \rangle \, \pi(\l) g^\circ \qquad \textnormal{ where } \qquad g^\circ = S_{g,\L}^{-1} g.
\end{equation}
Likewise, using the canonical dual window, we can expand any $f \in \Lt[]$ with respect to the elements of the Gabor system $\G(g,\L)$;
\begin{equation}
	f = \sum_{\l \in \L} \langle f , \pi(\l) g^\circ \rangle \, \pi(\l) g.
\end{equation}

A function space coming along in a natural way with the STFT is Feichtinger's algebra. Introduced by Feichtinger in the early 1980s \cite{Fei81}, it has become a popular function space in time-frequency analysis to choose the window $g$ from.
\begin{definition}[Feichtinger's Algebra]
	Feichtinger's algebra $S_0(\Rd)$ consists of all elements $g \in \Lt$ such that
	\begin{equation}
		\norm{\V_g g}_{L^1(\R^{2d})} = \iint_{\R^{2d}} \left| \V_g g(\l) \right| \, d\l < \infty, \qquad \l = (x,\omega) \in \R^{2d} .
	\end{equation}
\end{definition}
We note that $S_0(\Rd)$ is actually a Banach space, invariant under the Fourier transform and time--frequency shifts. It contains the Schwartz space $\mathcal{S}(\Rd)$ and it is dense in $\Lp{p}(\Rd)$, $p \in [1,\infty [$. It is for these properties that it is a quite popular function space in time--frequency analysis and the literature on the subject is huge. For more details on $S_0$ we refer to the survey by Jakobsen \cite{Jakobsen_S0_2018} and the references therein.

For $g \in S_0(\Rd)$, the upper bound in \eqref{eq_frame} is always finite which follows from the results of Tolimieri and Orr \cite{TolOrr95} (see also \cite{FeiZim98}). Hence, for windows in $S_0(\Rd)$, in order to check whether a Gabor system is a frame or not, it suffices to check whether the lower frame bound is positive or not. There are several equivalent approaches to check this, which usually involve (vector--valued) Zak transform methods and the Poisson summation formula. Functions in $S_0(\Rd)$ are at least continuous and there are no issues when it comes to integrability and summability. Also, Poisson's summation formula holds point-wise in this setting \cite{Gro_Poisson_1996}.

\section{The Symplectic and the Metaplectic Group}\label{sec_symp_meta}
In this section we are going to introduce the necessary tools to prove Theorem \ref{thm_main}. These tools are a group of matrices, the symplectic matrices, and a group of unitary operators, the metaplectic operators. There is a close connection between the symplectic and the metaplectic group and they are widely used in mathematical physics and quantum mechanics. In time-frequency analysis they can be used to prove deformation results about Gabor systems. For further reading we refer to the texbooks of Folland \cite{Fol89}, de Gosson \cite{Gos11,Gosson_Wigner_2017} or Gröchenig \cite{Gro01} and some more details are given in Appendix \ref{app}.

\begin{definition}[Generator Matrices]\label{def_gen_matrix}
	We define the following $2d \times 2d$ matrices which we call generator matrices for the symplectic group.
	\begin{itemize}
		\item[-] The standard symplectic matrix
		\begin{equation}
			J = \J,
		\end{equation}
		with $I$ being the $d \times d$ identity matrix.
		\item[-] The dilation matrices
		\begin{equation}
			M_L = \ML \label{eq_dilation_matrix}
		\end{equation}
		with $L$ being an invertible matrix, i.e., $\det(L) \neq 0$.
		\item[-] The shearing matrices
		\begin{equation}
			V_P = \VP \label{eq_lower_tri_matrix}
		\end{equation} 
		with $P$ being a real, symmetric matrix, i.e., $P = P^T$.
	\end{itemize}
\end{definition}
We note that any symplectic matrix $S \in Sp(d)$ is a finite product of the generator matrices from Definition \ref{def_gen_matrix} (see Appendix \ref{app}). In general, the group of symplectic matrices $Sp(d)$ is a proper subgroup of the special linear group $SL(\R,2d)$. Only if $d=1$, we have $Sp(1) = SL(\R,2)$.

\begin{definition}[Generator Operators]\label{def_gen_operator}
	We define the following unitary operators on $\Lt$ which we call the generator operators for the metaplectic group.
	\begin{itemize}
		\item[-] The modified Fourier transform
		\begin{equation}
			\Jhat g(t) = i^{-d/2} \F g(t)
		\end{equation}
		\item[-] The dilation operator
		\begin{equation}
			\Mhat_{L,m} \, g(t) = i^m \sqrt{|\det{L}|} \, g(L t), \quad \det{L} \neq 0, \; m \in \{0,1,2,3\}
		\end{equation}
		\item[-] The (linear) chirp
		\begin{equation}
			\Vhat_P \, g(t) = e^{\pi i Pt \cdot t} g(t), \quad P = P^T
		\end{equation}
		In the sequel we will write $P t^2$ instead of $Pt \cdot t$.
	\end{itemize}	 
\end{definition}
We note that any metaplectic operator $\widehat{S} \in Mp(d)$ is a finite product of the generator operators from Definition \ref{def_gen_operator}. More details are given in the appendix.

The integer $m$ is the so-called Maslov index (see e.g.~\cite{GosLue14}), but, just as the factor $i^{-d/2}$, it does in no way affect the frame property, as can be seen directly from the frame inequality \eqref{eq_frame}, and can simply be ignored for our purposes.

A proof of the following result can for example be found in \cite{Gos15}.
\begin{theorem}\label{thm_deform}
	The Gabor systems $\G(g,\L)$ and $\G\left(\widehat{S}g, S \L\right)$ possess the same sharp frame bounds. In particular, $\G(g,\L)$ is a frame with bounds $A$ and $B$ if and only if the Gabor system $\G\left(\widehat{S}g, S \L\right)$ is a frame with bounds $A$ and $B$.
\end{theorem}
The above theorem shows that the frame property is kept if, both, the lattice and the window are deformed in an appropriate way. Using the above result, it was shown in \cite{Faulhuber_Invariance_2016} that for a given 1-dimensional, generalized Gaussian there exists an uncountable family of lattices which all yield the same optimal frame bounds. This result is easily transferred to 1-dimensional Hermite functions. On the other hand, it was shown that for any lattice, there always exists a countable family of windows such that the optimal frame bounds are kept.

In this work, we will use Theorem \ref{thm_deform} to prove Theorem \ref{thm_main}. Before we get to the proof of Theorem \ref{thm_main}, we need another result involving metaplectic operators. For brevity, we will write $g(t) = \pm g(-t)$ if either $g(t) = g(-t)$ or $g(t) = -g(-t)$.
\begin{proposition}\label{pro_parity}
	Let $g \in S_0(\Rd)$, with the property that $g(t) = \pm g(-t)$ and let $\widehat{S} \in Mp(d)$, then $\widehat{S}g(t) \in S_0(\Rd)$ and $\widehat{S}g(t) = \pm \widehat{S}g(-t)$.
\end{proposition}
\begin{proof}
	Since $\widehat{S}$ is unitary, it is obvious that if $g \in S_0(\Rd)$ then $\widehat{S}g \in S_0(\Rd)$.
	
	The proof that metaplectic operators keep the parity of a function is elementary. Throughout, $c$ will denote the appropriate global phase factor as described in Definition \ref{def_gen_operator}.
	
	Then, we have
	\begin{equation}
		\Jhat g(t) = c \, \int_{\Rd} \pm g(-t') e^{-2 \pi i t \cdot t'} \, dt' = c \, \int_{\Rd} \pm g(t') e^{2 \pi i t \cdot t'} \, dt' = \pm \Jhat g(-t),
	\end{equation}
	\begin{equation}
		\Mhat_{L,m} g(t) = c \, \sqrt{|\det{L}}| g(L t) = \pm c \, \sqrt{|\det(L)|} g(-L t) = \pm \Mhat_{L,m} g(-t)
	\end{equation}
	and
	\begin{equation}
		\Vhat_P g(t) = c \, e^{\pi i P t^2} g(t) = \pm c \, e^{\pi i P t^2} g(-t) = \pm \Vhat_P g(-t).
	\end{equation}
	The result follows from the fact that any metaplectic operator $\widehat{S} \in Mp(d)$ is a finite composition of the above operators (more details are given in Appendix \ref{app}).
\end{proof}

We are now able to prove Theorem \ref{thm_main}.

\subsection*{Proof of Theorem \ref{thm_main}}
Assume that $g \in S_0(\Rd)$ is either an even or odd function. As already stated, for a window in $S_0(\Rd)$ the frame property solely depends on the lower frame bound. Let $\delta > 1$ be fixed, and consider the Gabor system
\begin{equation}
	\G(g, \delta^{-1/2d} \, \Z^{2d})
\end{equation}
which we assume (not) to be frame for all $g \in S_0(\Rd)$ which are either even or odd. It follows by assumption and Proposition \ref{pro_parity} that the Gabor system
\begin{equation}
	\G \left( \widehat{S}^{-1} g, \delta^{-1/2d} \, \Z^{2d} \right)
\end{equation}
is also (not) a frame. By using Theorem \ref{thm_deform} we see that the systems
\begin{equation}
	\G \left( \widehat{S}^{-1} g, \delta^{-1/2d} \, \Z^{2d} \right) \qquad \textnormal{ and } \qquad \G (g , \delta^{-1/2d} \, S \Z^{2d})
\end{equation}
possess the same sharp frame bounds. Hence, it follows that if for any either even or odd $g \in S_0(\Rd)$ the system $\G(g, \delta^{-1/2d} \, \Z^{2d})$ is (not) a frame, then $\G (g , \delta^{-1/2d} \, S \Z^{2d})$ is (not) a frame and this is true for any $S \in Sp(d)$.
\begin{flushright}
	$\square$
\end{flushright}

We remark that, in Theorem \ref{thm_main}, the parity of $g$ can be exchanged for any other property of the window which stays invariant under metaplectic operators. Also, we focused on the function space $S_0(\Rd)$, as the motivation was to extend the results of Lyubarskii and Nes \cite{LyubarskiiNes_Rational_2013}. However, using the continuity results of metaplectic operators in, e.g., \cite{CorNic08JFA} \cite{CorNic08JDE}, extensions of Theorem \ref{thm_main} to Wiener amalgam spaces should be possible. As the results in \cite{LyubarskiiNes_Rational_2013} were only stated for $S_0(\R)$ ($d$=1), we avoided technical details by assuming $g \in S_0(\Rd)$.

\section{Gabor Systems on the Line at Density \texorpdfstring{$\frac{n+1}{n}$}{(n+1)/n}}\label{sec_density}
Now, we restrict the setting to $\Lt[]$, i.e., $d$ = 1, and generalize the main result of Lyubarskii and Nes.
\begin{theorem}[Lyubarskii, Nes \cite{LyubarskiiNes_Rational_2013}]\label{thm_LyuNes}
	Let $g \in S_0(\R)$ be an odd function, i.e., $g(t) = -g(-t)$ and let $M_L \in SL(\R,2)$ be a dilation matrix as defined in Definition \ref{def_gen_matrix}. Then, for $\delta \in \{ \tfrac{n+1}{n}, n \in \N\}$ the Gabor system $\G \left(g, \delta^{-1/2} M_L \Z^2 \right)$ is not a frame.
\end{theorem}
By applying Theorem \ref{thm_main} we get the following result.
\begin{corollary}\label{cor_equal}
	Let $g \in S_0(\R)$ be an odd function, i.e., $g(t) = -g(-t)$, and let
	\begin{equation}
		\L = \delta^{-1/2} S \Z^2, \qquad S \in Sp(1) = SL(\R,2),
	\end{equation}
	be a lattice of density $\delta$. The, the following are equivalent
	\begin{enumerate}[(i)]
		\item The Gabor system $\G(g, \delta^{-1/2} \Z^2)$ is not a frame for $\delta \in \{ \tfrac{n+1}{n}, n \in \N \}$.\label{thm_equal_square}
		\item The Gabor system $\G(g, \delta^{-1/2} M_L \Z^2)$ is not a frame for $\delta \in \{ \tfrac{n+1}{n}, n \in \N\}$.\label{thm_equal_rect}
		\item The Gabor system $\G(g, \delta^{-1/2} S \Z^2)$ is not a frame for $\delta \in \{ \tfrac{n+1}{n}, n \in \N\}$.\label{thm_equal_symp}
	\end{enumerate}
\end{corollary}
\begin{proof}
	The implications \eqref{thm_equal_symp} $\Rightarrow$ \eqref{thm_equal_rect} $\Rightarrow$ \eqref{thm_equal_square} are obvious. The implication \eqref{thm_equal_square} $\Rightarrow$ \eqref{thm_equal_symp} follows by combining Theorem \ref{thm_LyuNes} and Theorem \ref{thm_main}.
\end{proof}

We remark that we included \ref{cor_equal} \eqref{thm_equal_rect} in the above list as results about separable Gabor systems, which are derived under the action of the (symplectic) subgroup of dilation matrices $M_L$, are found more often in the literature than results on Gabor systems with general (symplectic) lattices.

Numerical experiments in \cite{LyubarskiiNes_Rational_2013} lead to the conjecture that, besides Theorem \ref{thm_LyuNes}, for the first Hermite function and separable lattice there are no other obstructions which prevent the resulting Gabor system from being a frame. By a suitable reformulation of Theorem \ref{thm_main}, we actually only have to consider square lattices. The suitable reformulation is to replace the property odd by the property of $g$ being a dilated first Hermite function, as this property is kept under the metaplectic operators in question (i.e., dilation operators). We will discuss more general, but related questions in Section \ref{sec_problems}.

\section{Gabor Systems at Critical Density}\label{sec_critical}
There are many proofs that a Gabor system with a window in $S_0(\Rd)$ cannot constitute a frame at critical density, i.e., $\delta = 1$. There is a variety of Balian-Low type theorems or density theorems, also holding for irregular Gabor frames \cite{AscFeiKai14}, which means that the index set is not a lattice. The correct notion of density is then to consider the lower Beurling density of the index set. We also refer to \cite{BenHeiWal_BLT_98}, \cite{Gro01}, or \cite{Hei07} and the references in \cite{AscFeiKai14} for more details.

By using Theorem \ref{thm_main}, we find a quick proof for a Balian-Low type theorem for even and odd functions in $S_0(\R)$ (which is of course only a special case of the results in \cite{AscFeiKai14}). The proof is similar to the proof in \cite{Faulhuber_Note_2018}, where it was shown that we cannot have a Gabor frame for an odd function $g \in S_0(\Rd)$ and a symplectic lattice of density $2^d$. Recall that the cross-ambiguity function of $f$ and $g$ is given by
\begin{equation}
	A_g f(x, \omega) = \int_{\Rd} f(t+\tfrac{x}{2}) \overline{g(t-\tfrac{x}{2})} e^{-2 \pi i \omega \cdot t} \, dt = e^{\pi i x \cdot \omega} V_g f (x, \omega).
\end{equation}
Closely related to the ambiguity function is the Wigner distribution, given by
\begin{equation}
	W_g f(x, \omega) = \int_{\Rd} f(x+\tfrac{t}{2}) \overline{g(x-\tfrac{t}{2})} e^{-2 \pi i \omega \cdot t} \, dt = 2^d A_{g^\vee} f(2x, 2\omega),
\end{equation}
where $g^\vee(t) = g(-t)$ is the reflection of $g$. Furthermore, the symplectic Fourier transform of a $2d$-dimensional function $F(\l)$, $\l = (x,\omega)$, is given by
\begin{align}
	\F_\sigma F(x, \omega) & = \iint_{\R^{2d}} F(x',\omega') e^{2 \pi i (x \cdot \omega' - \omega \cdot x')} \, d(x', \omega') = \int_{\R^2} F(\l') e^{-2 \pi i \, \sigma(\l, \l')} \, d \l'\\
	& = \F F(-\omega, x).
\end{align}
Here, $\F$ denotes the usual (planar) Fourier transform and $\sigma$ is the standard symplectic form, defined in Appendix \ref{app}. The ambiguity function and the Wigner transform are symplectic Fourier transforms of one another, i.e.,
\begin{equation}
	\F_\sigma \left( A_gf  \right)(\l) = W_g f (\l) \qquad \textnormal{ and } \qquad \F_\sigma \left( W_g f \right)(\l) = A_g f (\l).
\end{equation}

Consider the case $d =1 $ and recall the following proposition by Janssen \cite{Jan96} (see also \cite{Faulhuber_Note_2018}).
\begin{proposition}[Janssen \cite{Jan96}]\label{pro_Janssen}
	Let $g \in \Lt[]$ and $\alpha, \beta \in \R_+$ with $(\alpha \beta)^{-1} = \delta \in \N$. Assume that
	\begin{equation}
		\sum_{k,l \in \Z} \left| V_g g \left( \tfrac{k}{\beta}, \tfrac{l}{\alpha} \right) \right| < \infty.
	\end{equation}
	Then, the Gabor system $\G(g, \alpha \Z \times \beta \Z)$ possesses the sharp frame bounds
	\begin{equation}
		A = \essinf_{(x,\omega) \in \R^{2}} (\alpha \beta)^{-1} \sum_{k,l \in \Z} V_g g \left( \tfrac{k}{\beta}, \tfrac{l}{\alpha} \right) e^{2 \pi i (k \omega + l x)}
	\end{equation}
	and
	\begin{equation}
		B = \esssup_{(x,\omega) \in \R^{2}} (\alpha \beta)^{-1} \sum_{k,l \in \Z} V_g g \left( \tfrac{k}{\beta}, \tfrac{l}{\alpha} \right) e^{2 \pi i (k \omega + l x)}.
	\end{equation}
\end{proposition}

Using Janssen's result, we can quickly show that an even function belonging to $S_0(\R)$ cannot constitute a Gabor frame with a lattice of critical density.
\begin{proposition}
	Let $g \in S_0(\R)$ be an even function and consider the Gabor system $\G(g, \L)$ with $\L = S \Z^2$, i.e., $\delta = 1$. Then the lower frame bound vanishes and, hence, the Gabor system is not a frame. 
\end{proposition}
\begin{proof}
	Consider the lattice $\Z^2$ ($\alpha = \beta = 1$) and in the series in Proposition \ref{pro_Janssen} set $(x, \omega) = \left(\tfrac{1}{2}, \tfrac{1}{2} \right)$. We set
	\begin{equation}
		\widetilde{A} = \sum_{k,l \in \Z} V_g g (k,l) e^{\pi i (k+l)} = \sum_{k,l \in \Z} e^{-\pi i k l} A_g g (k,l) e^{\pi i (k+l)} = \sum_{k,l \in \Z} (-1)^{k+kl+l} A_g g (k,l).
	\end{equation}
	Now, we split the last expression above in the part where the sign is positive and where it is negative. We see that $(-1)^{k+kl+l}$ is positive only if both $k$ and $l$ are even. So, we also get the value $\widetilde{A}$ by summing over the even integers twice and then subtracting the sum over all integers;
	\begin{equation}
		\widetilde{A} = \sum_{k,l \in \Z} (-1)^{k+kl+l} A_g g (k,l) = \sum_{k,l \in \Z} 2 A_g g (2k,2l) - \sum_{k,l \in \Z} A_g g (k,l).
	\end{equation}
	By the algebraic relation of the Wigner and ambiguity function and due to the fact that $g^\vee = g$, we have $2A_g g(2k,2l) = W_g g(k,l)$, which gives
	\begin{equation}
		\widetilde{A} = \sum_{k,l \in \Z} W_g g (k,l) - \sum_{k,l \in \Z} A_g g (k,l).
	\end{equation}
	By using a version of the Poisson summation involving the symplectic Fourier transform (see also \cite{Faulhuber_Note_2018}) we get that
	\begin{equation}
		\sum_{k,l \in \Z} W_g g (k,l) = \sum_{k,l \in \Z} A_g g (k,l).
	\end{equation}
	Therefore, it follows that $\widetilde{A} = 0$, which then has to be the sharp lower frame bound. This shows that the lower frame bound of the Gabor system $\G(g, \Z^2)$ indeed vanishes. By using Theorem \ref{thm_main}, it follows that the Gabor system $\G(g, \L)$, $\L = S \Z^2$, is not a frame for any lattice of critical density.
\end{proof}
As remarked in the introduction, due to the results in \cite{Lyu92}, \cite{Sei92} and \cite{SeiWal92} on Gaussian windows, this is the only example of a fixed density where we can completely exclude all lattices of a given density from the frame set of all even functions from $S_0(\R)$.

In a similar way, we can establish the following result.
\begin{proposition}\label{pro_odd}
	Let $g \in S_0(\R)$ be an odd function and consider the Gabor system $\G(g, \L)$ with $\L = S \Z^2$, i.e., $\delta = 1$. Then the lower frame bound vanishes and, hence, the Gabor system is not a frame. 
\end{proposition}
\begin{proof}
	Again, consider the lattice $\Z^2$ and the series in Proposition \ref{pro_Janssen}, but this time set $(x, \omega) = \left(0,0 \right)$. This leads to the expression
	\begin{equation}
		\widetilde{A} = \sum_{k,l \in \Z} V_g g (k,l) = \sum_{k,l \in \Z} e^{-\pi i k l} A_g g (k,l) = \sum_{k,l \in \Z} (-1)^{kl} A_g g (k,l).
	\end{equation}
	Now, we split the last expression above in the following way, which we will justify right away;
	\begin{align}
		\widetilde{A} & = \sum_{k,l \in \Z} (-1)^{kl} A_g g (k,l)\\
		& = \sum_{k,l \in \Z} 2 A_g g (2k,l) + \sum_{k,l \in \Z} 2 A_g g (k,2l) - \sum_{k,l \in \Z} 2 A_g g (2k,2l) - \sum_{k,l \in \Z} A_g g (k,l).
	\end{align}
	The first two sums do not count lattice points where both indices are odd. Points with one even and one odd index are counted twice by each of the first two sum. Points with two even indices are also counted twice by each of the first two sums, so in total they are counted 4 times. Then we subtract all even lattice points twice and finally subtract all lattice points once, which means that all lattice points with both indices being odd have a negative sign and all other points have a positive sign.
	
	This time, as $g^\vee = - g$, we note that
	\begin{equation}
		\sum_{k,l \in \Z} \underbrace{2 A_g g (2k,2l)}_{= - W_g g(k,l)} + \sum_{k,l \in \Z} A_g g (k,l) = 0.
	\end{equation}
	This leaves us with
	\begin{equation}
		\widetilde{A} = \sum_{k,l \in \Z} (-1)^{kl} A_g g (k,l) = \sum_{k,l \in \Z} 2 A_g g (2k,l) + \sum_{k,l \in \Z} 2 A_g g (k,2l)
	\end{equation}
	By using the (symplectic version of) the Poisson summation formula first and the algebraic relation between the Wigner transform and the ambiguity function afterwards, we get
	\begin{equation}
		\sum_{k,l \in \Z} 2 A_g g (2k,l) = \sum_{k,l \in \Z} W_g g \left(k, \tfrac{l}{2} \right) = -\sum_{k,l \in \Z} 2 A_g g (2k,l) = 0.
	\end{equation}
	With the same trick we also see that
	\begin{equation}
		\sum_{k,l \in \Z} 2 A_g g (k,2l) = 0.
	\end{equation}
	Hence, $\widetilde{A} = 0$, which means that the lower frame bound of the Gabor system $\G(g,\Z \times \Z)$ vanishes. Now, use Theorem \ref{thm_main} to establish the general result.
\end{proof}
We note that splitting the sum as we did in the proof of Proposition \ref{pro_odd}, already gives a proof that for an odd window in $S_0(\R)$ we cannot have a Gabor frame for a lattice of density 2, as showing that
\begin{equation}
	\sum_{k,l \in \Z} 2 A_g g (2k,l) = 0
\end{equation}
is, by Janssen's result \ref{pro_Janssen}, equivalent to showing that the Gabor system $\G(g, \Z \times \tfrac{1}{2} \Z)$ is not a frame. Interestingly, the proof of Proposition \ref{pro_odd} already covers the cases $n = 1$ (density 2) and $n = \infty$ (density 1) in Theorem \ref{thm_LyuNes}.

Of course, we know that the Balian-Low theorem holds for any $g \in S_0(\R)$ (even for general index sets), but it is due to the parity of the window that we could find the 0 of the series
\begin{equation}
	\sum_{k,l \in \Z} V_g g \left( \tfrac{k}{\beta}, \tfrac{l}{\alpha} \right) e^{2 \pi i (k \omega + l x)}
\end{equation}
by a ``lucky guess". Therefore, this simple proof may probably not work for a generic window in $S_0(\R)$.

\section{Open Problems}\label{sec_problems}
Before closing this work, we define the following (probably incomplete) list of frame sets, which extends the list given in \cite{Gro14}. We note that the lattice density is now hidden in the notation used for lattices. For a fixed dimension $d \in \N$, we define:
\begin{itemize}
	\item[-] The $\alpha$-frame set
	\begin{equation}
		\mathfrak{F}_\alpha (g) = \lbrace \alpha \Z^{2d} \mid \G(g, \alpha \Z^{2d}) \textnormal{ is a frame.} \rbrace.
	\end{equation}
	\item[-] The separable or $(\alpha,\beta)$-frame set
	\begin{equation}
		\mathfrak{F}_{(\alpha, \beta)} (g) = \lbrace \alpha \Z^d \times \beta \Z^d \mid \G(g, \alpha \Z^d \times \beta \Z^d) \textnormal{ is a frame.} \rbrace.
	\end{equation}
	For $d > 1$ we allow $\alpha$ and $\beta$ to be multi-indices and $\alpha \Z^d = \alpha_1 \Z \times \ldots \times \alpha_d \Z$ (same for $\beta$). In \cite{Gro14} this set was defined for $d =1 $ and called the reduced frame set.
	\item[-] The symplectic or $\sigma$-frame set
	\begin{equation}
		\mathfrak{F}_{\sigma} (g) = \lbrace \L_\sigma \subset \R^{2d} \textnormal{ symplectic lattice} \mid \G(g, \L_\sigma) \textnormal{ is a frame.} \rbrace.
	\end{equation}
	A symplectic lattice is a lattice that can be generated by a symplectic matrix $S \in Sp(d)$;
	$\L_\sigma = \delta^{-1/2d} S \Z^{2d}$, $\delta > 0$.
	\item[-] The lattice or $\L$-frame set
	\begin{equation}
		\mathfrak{F}_{\L} (g) = \lbrace \L \subset \R^{2d} \textnormal{ lattice} \mid \G(g, \L) \textnormal{ is a frame.} \rbrace.
	\end{equation}
	In \cite{Gro14} this set was defined for $d = 1$ and called the full frame set.
	\item[-] The frame set
	\begin{equation}
		\mathfrak{F} (g) = \lbrace \Gamma \subset \R^{2d} \mid \G(g, \Gamma) \textnormal{ is a frame.} \rbrace.
	\end{equation}
\end{itemize}
We have the following chain of inclusions for the defined frame sets;
\begin{equation}
	\mathfrak{F}_\alpha \subset \mathfrak{F}_{(\alpha, \beta)} \subset \mathfrak{F}_\sigma \subset \mathfrak{F}_{\L} \subset \mathfrak{F}
\end{equation}
For $d = 1$, we have that $\mathfrak{F}_\sigma = \mathfrak{F}_\L$. By using Theorem \ref{thm_main} we were able to generalize the result of Lyubarskii and Nes from $\mathfrak{F}_{\alpha}$ directly to $\mathfrak{F}_\sigma = \mathfrak{F}_\L$ (although the result was already established for $\mathfrak{F}_{(\alpha, \beta)}$ we only needed it for $\mathfrak{F}_\alpha$).

This insight leads to the following question: Does the result in Corollary \ref{cor_LyuNes} depend on the group structure of $\L$ or does the result hold for any relatively separated, discrete set $\Gamma \subset \R^2$ with lower Beurling density $\tfrac{n+1}{n}$? At the moment, the author has no clue in either direction and, therefore, considers the question as an interesting open problem.

Also, for $d > 1$, the results in \cite{Faulhuber_Note_2018} show that we cannot have a Gabor frame consisting of an odd window $g \in S_0(\Rd)$ and a symplectic lattice of denisty $2^d$. Is it possible to extend the result to arbitrary lattices or more general point sets of lower Beurling density $2^d$? Lastly, the following question comes up; how generic is the result of Lyubarskii and Nes truly? Does a higher dimensional result exist for densities $\left( \tfrac{n+1}{n} \right)^d$?

As a last remark, we mention that, by finding appropriate properties defining a class of windows invariant under the action of the metaplectic group, versions of Theorem \ref{thm_main} seem to be particularly helpful to completely exclude lattices of a certain densities from the frame set of a whole class of functions.

\begin{appendix}
	\section{Properties of the Symplectic and the Metaplectic Group}\label{app}
	For background information and motivation of the following definitions as well as for the proofs of the results in this section we refer to the textbook of de Gosson \cite{Gos11}. We start with the definition of a symplectic matrix.
	\begin{definition}[Symplectic Matrix]\label{def_symplectic_matrix}
		A matrix $S \in GL(2d,\R)$ is called {symplectic} if and only if
		\begin{equation}\label{eq_symplectic_matrix}
			S J S^T = S^T J S = J.
		\end{equation}
	\end{definition}
	An equivalent definition involves the so-called symplectic form, which plays a central role in symplectic geometry and Hamiltonian mechanics. For $z = (x,\omega), \, z' =(x',\omega') \in \R^{2d}$, the symplectic form $\sigma$ is given by
	\begin{equation}
		\sigma(z,z') = x \cdot \omega' - x' \cdot \omega,
	\end{equation}
	where the dot $\cdot$ denotes the Euclidean inner product in $\Rd$. A matrix $S$ is symplectic if and only if
	\begin{equation}
		\sigma(z,z') = \sigma(Sz,Sz').
	\end{equation}		
	This definition is equivalent to Definition \ref{def_symplectic_matrix}. In symplectic geometry, symplectic matrices take a role similar to that of orthogonal matrices in Euclidean geometry;
	\begin{align}
		M \text{ is orthogonal } & \Longleftrightarrow \; M^T I M = I\\
		S \text{ is symplectic } & \Longleftrightarrow \; S^T J S = J.
	\end{align}
	We refer to \cite{Gos13} for an enjoyable introduction to symplectic matrices and their use in Hamiltonian mechanics.
		
	It is not hard to show that symplectic matrices actually form a group under matrix multiplication. Also, if $S \in Sp(d)$, then $\det(S) = 1$. The following decomposition property of symplectic matrices was implicitly used in this work.
	\begin{proposition}\label{pro_symplectic_factorization}
		Let $S = \Sblock$ be a symplectic matrix with $d \times d$ blocks and the property that $\det(B) \neq 0$. Then, $S$ can be factored as
		\begin{equation}\label{eq_factor_symplectic_matrix_1}
			S = V_{{DB}^{-1}}M_{B^{-1}}JV_{B^{-1}A}.
		\end{equation}
	\end{proposition}
	Such a matrix is called a free symplectic matrix. The following theorem now shows that any symplectic matrix can be decomposed into the generator matrices from Section \ref{sec_symp_meta}.
	\begin{theorem}
		Any matrix $S \in Sp(d)$ is the (non-unique) product of exactly two free symplectic matrices $S_1$ and $S_2$.
	\end{theorem}
	
	We turn our attention to the metaplectic group. One way to define it, is to say that $Mp(d)$ is the connected two--fold cover of $Sp(d)$ or that the following sequence is exact;
	\begin{equation}
		0 \rightarrow \Z_2 \rightarrow Mp(d) \rightarrow Sp(d) \rightarrow 0.
	\end{equation}
	Hence, we have the following identification
	\begin{equation}
		Sp(d) \cong Mp(d)\slash \lbrace\pm I \rbrace.
	\end{equation}
	There is also another, less abstract way to define the metaplectic group. This approach makes use of the generating operators from Section \ref{sec_symp_meta}.
	\begin{definition}[Quadratic Fourier Transform]\label{definition_quadratic_FT}
		Let $S_W$ be the free symplectic matrix
		\begin{equation}
			S_W = \SW
		\end{equation}
		associated to the quadratic form $W(t,t') = \frac{1}{2} P t^2 - L t \cdot t' + \frac{1}{2} Q t'^2$, called the generating function of the (free) symplectic matrix $S_W$. The operator
		\begin{equation}\label{eq_quadratic_FT}
			\widehat{S}_{W,m} = \Vhat_P \Mhat_{L,m} \Jhat \Vhat_Q
		\end{equation}
		is called the quadratic Fourier transform associated to the free symplectic matrix $S_W$ ($m$ denotes the Maslov index from Definition \ref{def_gen_operator}).
	\end{definition}
	For $g \in S_0(\Rd)$ we have the explicit formula
	\begin{equation}\label{eq_quadratic_FT_formula}
		\widehat{S}_{W,m} \, g \left( t \right) = i^{m-\frac{d}{2}} \sqrt{|\det(L)|} \int_{\Rd}  g(t) \, e^{2 \pi i \, W\left(t,t'\right)} \, dt'.
	\end{equation}
	Just as in the case of the symplectic group, there is a (non--unique) way of factorizing a metaplectic operator into two quadratic Fourier transforms.
	\begin{theorem}\label{thm_factorization_meta}
		For every $\widehat{S} \in Mp(d)$ there exist two quadratic Fourier transforms $\widehat{S}_{W_1,m_1}$ and $\widehat{S}_{W_2,m_2}$ such that $\widehat{S} = \widehat{S}_{W_1,m_1} \widehat{S}_{W_2,m_2}$.
	\end{theorem}
	There is a natural projection from the metaplectic group $Mp(d)$ onto the symplectic group $Sp(d)$, which we will denote by $\pi^{Mp}$. The following theorem lies at the heart of the so-called Hamiltonian deformation of the Gabor system $\G(g,\L)$ as used in Theorem \ref{thm_deform}.
	\begin{theorem}\label{thm_natural_projection}
		The mapping
		\begin{equation}
			\begin{aligned}
				\pi^{Mp}: & & Mp(d) & \longrightarrow Sp(d)\\
				& & \widehat{S}_{W,m} & \longmapsto S_W
			\end{aligned}
		\end{equation}
		which associates a free symplectic matrix with generating function $W$ to a quadratic Fourier transform, is a surjective group homomorphism. Hence,
		\begin{equation}
			\pi^{Mp} \left( \widehat{S}_1 \widehat{S}_2 \right) = \pi^{Mp} \left( \widehat{S}_1 \right) \pi^{Mp} \left( \widehat{S}_2 \right).
		\end{equation}
		and the kernel of $\pi^{Mp}$ is given by
		\begin{equation}
			ker(\pi^{Mp}) = \lbrace \pm I \rbrace.
		\end{equation}
		Therefore, $\pi^{Mp}: Mp(d) \mapsto Sp(d)$ is a two-fold covering of the symplectic group.
	\end{theorem}
\end{appendix}

%\bibliographystyle{plain}
%\bibliography{../../mybib}

\end{document}